\documentclass[12pt,twoside,reqno]{amsart}
\usepackage[colorlinks=true,citecolor=blue]{hyperref}
\usepackage{mathptmx, amsmath, amssymb, amsfonts, amsthm, mathptmx, enumerate, color,mathrsfs}
\usepackage{graphicx}

\usepackage{epstopdf}

\newtheorem{theorem}{Theorem}[section]

\newtheorem{proposition}[theorem]{Proposition}

\theoremstyle{definition}
\newtheorem{definition}[theorem]{Definition}

\newtheorem{example}[theorem]{Example}

\newtheorem{remark}[theorem]{Remark}
\numberwithin{equation}{section}

\begin{document}
\setcounter{page}{1}

\vspace*{2.0cm}
\title[Variation arising from a generalized two-normed space]
{Functions of bounded variation arising from a generalized two-normed space}
\author[F. Osmin, F. Osmin, C. Jaffeth]{ Osmin Ferrer Villar$^{1,*}$, Kandy Ferrer Sotelo$^2$, Jaffeth Cure Arenas$^3$}
\maketitle
\vspace*{-0.6cm}

\begin{center}
{\footnotesize

$^1$Department of Mathematics, University of Sucre.
Red door, Sincelejo, Sucre, Colombia.\\
$^2$Center for Basic Sciences, School of Engineering and Architecture, Pontifical Bolivarian University, Monteria, Colombia.\\
$^3$Department of Applied Mathematics and Physics, Catholic University of the Most Holy Conception, Concepción, Chile.
}\end{center}

\vskip 4mm {\footnotesize \noindent {\bf Abstract.}
The first formal development of functions with bounded variation in normed spaces is attributed to Chistyakov \cite{Chistyakov Metric}, and was later extended to the context of 2-normed spaces by Cure, Ferrer S., and Ferrer V. \cite{CFF}. In this paper, we elaborate on this extension (Definition \ref{vaacotada}), exploring its fundamental properties. The foundational characteristics of functions of bounded 2-variation within the context of generalized two-normed spaces are explored in detail (see Theorems \ref{ass}, \ref{sum prod}, \ref{distributive}).  Additionally, it is shown that a function of bounded variation defined on a semi-normed space can generate a function of bounded 2-variation in the context of generalized two-normed spaces (Proposition \ref{varimplica2kacotada}).


 }

\renewcommand{\thefootnote}{}
\footnotetext{ $^*$Corresponding author.
\par
E-mail addresses: osmin.ferrer@unisucre.edu.co (F. Osmin), kandy.ferrer@upb.edu.co (F. Kandy),\\ jcure@magister.ucsc.cl (C. Jaffeth).
\par
Received January 23, 2015; Accepted February 16, 2016. }

\section{Introduction}
\label{Sec:1}
Chistyakov, in \cite{Chistyakov mappings, cmultivalued, Chistyakov Metric}, conducted a detailed study of the concept of variation for functions defined on a subset $E \subset \mathbb{R}$, whose values belong to a normed or metric space, using the notion of $p$-variation in the sense of Wiener, as well as the classical variation introduced by Jordan \cite{Jordan}, among other approaches. From the need to extend the scope of bounded variation functions, extensions towards more general structures \cite{FNG, CFF} have emerged. In this paper, we propose a new perspective on these ideas that generalizes these notions within the setting of generalized two-formed spaces.

The notion of bounded variation has proven essential in various practical applications, such as image processing, image restoration, and noise removal. Examples include the high-order total variation model proposed by Banothu \cite{Banothu}, which enhances gamma noise removal while preserving details without staircase effects; the curvature-dependent model for image restoration by Nan, Qiao, and Zhang \cite{Nan}, which maintains edge sharpness through elastic energy and total variation; and the approach by Bi, Li, Cai, and Zhang \cite{Bi}, which introduces a non-convex variation for image reconstruction, optimizing both visual and quantitative results in complex images. These applications highlight the importance of studying functions of bounded variation within more general frameworks.

We guarantee that the concept of bounded variation function is well defined in the structure of generalized two-normed spaces which constitutes a significant advance in the discovery of new applications of this theory in more general contexts. 

\section{Research gap}
Although Chistyakov \cite{Chistyakov Metric} and Cure, Ferrer S., and Ferrer V. \cite{CFF} have made significant advancements in normed and two-normed spaces, respectively, there has been no comprehensive exploration of how these definitions can be generalized or adapted to generalized two-normed spaces, which provide a broader and more flexible framework for functional analysis. Despite their wide application in areas such as the theory of operators \cite{lewandowska1}, the study of functions of bounded variation in the context of generalized two-normed spaces remains largely unexplored, representing a significant lacune in the field.

The classical theory of functions of bounded variation has proven crucial in practical applications such as image restoration and noise removal, where variation controls the smoothness of solutions. However, understanding how these functions behave in generalized two-normed spaces could unlock new possibilities for modeling and analyzing these problems. Despite the well-established applications of functions of bounded variation, the absence of a rigorous framework for their study in these spaces limits the potential to extend classical results to more general and complex contexts.

We attempt to fill this gap by investigating two fundamental questions:

(i) How can one adequately define the notion of functions of bounded variation in generalized two-normed spaces, while respecting their fundamental characteristics?

(ii) What conditions must be fulfilled for functions of bounded variation in generalized two-formed spaces to retain the essential properties of classical functions of bounded variation?

\section{Preliminaries}
\begin{definition}\cite{lewandowska1, lewandowska}
Let $A, B$ be vector spaces over $\mathbb{C}$.\\
$\mathbb{W}\subset A\times B$ will represent a non-empty set such that for each $a \in A$, $b \in B$ the following characteristics are satisfied.
\begin{itemize}
    \item $\mathbb{W}^{b} = \{a \in A : (a, b) \in \mathbb{W}\}$ is a vector subspace of $A$.
    \item $\mathbb{W}_{a} = \{b \in B : (a, b) \in \mathbb{W}\}$ is a vector subspace of $B$.
\end{itemize}
The mapping $\|\cdot, \cdot\| \colon \mathbb{W} \longrightarrow [0, \infty)$ which satisfies:
\begin{enumerate}
\item[(G1)] $\|a, \alpha b\| = |\alpha|\|a, b\| = \|\alpha a, b\|$ for each $\alpha \in \mathbb{C}$ and each $(a, b) \in \mathbb{W}$;

\item[(G2)] $\|a, b + c\| \leq \|a, b\| + \|a, c\|$ for $a \in A$, $b, c \in B$ such that $(a, b), (a, c) \in \mathbb{W}$;

\item[(G3)] $\|a + b, c\| \leq \|a, c\| + \|b, c\|$ for $a, b \in A$, $c \in B$ such that $(a, c), (b, c) \in \mathbb{W}$.
\end{enumerate}
is known as a \emph{generalized two-norm on $\mathbb{W}$}.

When $\mathbb{W}$ coincides with $A\times B$, then $(A\times B, \|\cdot, \cdot\|)$ is said to be a generalized two-normed space. If, in addition, the two vector spaces $A,B$ are equal, we simply write $(A, \|\cdot, \cdot\|)$.

\end{definition}

\begin{example}\label{generaltwo-norma}
Consider $A = \mathbb{C}^2$ and $B = \mathbb{C}$. Let $\mathbb{W} = \mathbb{C}^2 \times \mathbb{C}$ be a set such that for every $k \in \mathbb{C}$ and $w \in \mathbb{C}^2$, the sets $\mathbb{W}^k = \{w \in \mathbb{C}^2 : (w, k) \in \mathbb{W}\}$ and $\mathbb{W}_w = \{k \in \mathbb{C} : (w, k) \in \mathbb{W}\}$ are vector subspaces of $\mathbb{C}^2$ and $\mathbb{C}$, respectively.

The application $\|\cdot, \cdot\| \colon \mathbb{W} \longrightarrow [0, \infty)$, defined by $\|w, k\| = \|(w_1, w_2), k\| = \sqrt{|w_1|^2 + |w_2|^2} |k|$ is a generalized two-norm on $\mathbb{W}$.

In fact,

(G1) Let $w \in \mathbb{C}^2$, $k \in \mathbb{C}$, and $\beta \in \mathbb{C}$ such that $(w, k) \in \mathbb{W}$.
\begin{align*}
\|w, \beta k\| &= \sqrt{|w_{1}|^2 + |w_{2}|^2} |\beta k| = |\beta| \sqrt{|w_{1}|^2 + |w_{2}|^2} |k| = |\beta| \|w, k\| \\
&= \sqrt{|\beta|^2 (|w_{1}|^2 + |w_{2}|^2)} |k| = \sqrt{|\beta w_{1}|^2 + |\beta w_{2}|^2} |k| = |\beta w, k\|.
\end{align*}

(G2) Let $w \in \mathbb{C}^2$ and $k, l \in \mathbb{C}$ such that $(w, k), (w, l) \in \mathbb{W}$.
\begin{align*}
\|w, k + l\| &= \sqrt{|w_{1}|^2 + |w_{2}|^2} |k + l| \le \sqrt{|w_{1}|^2 + |w_{2}|^2} |k| + \sqrt{|w_{1}|^2 + |w_{2}|^2} |l| \\
&= \|w, k\| + \|w, l\|.
\end{align*}

(G3) Let $w, l \in \mathbb{C}^2$ and $k \in \mathbb{C}$ such that $(w, k), (w, k) \in \mathbb{W}$.
\begin{align*}
\|w + l, k\| &= \sqrt{|w_{1} + l_{1}|^2 + |w_{2} + l_{2}|^2} |k| \\
&\le \sqrt{|w_{1}|^2 + |w_{2}|^2} |k| + \sqrt{|l_{1}|^2 + |l_{2}|^2} |k| = \|w, k\| + \|l, k\|.
\end{align*}

\end{example}

\begin{definition}\cite{CFF}
Given $(\mathbb{W}\subset A\times B,\|\cdot,\cdot\|)$ a generalized two-normed space and $k\in B$. A mapping $g\colon [a,b]\longrightarrow A$ is defined \emph{$(2,k)$-bounded} if the set $\big\{\|g(t),k\|\colon t\in [a,b]\big\}$ is upper bounded.

\end{definition}

\begin{remark}\cite{lewandowska1}
Let $A \neq \emptyset$ be a set and $\mathrm{E} \subset A\times A$. If
$$
(a, b) \in \mathrm{E} \Longrightarrow (b, a) \in \mathrm{E}, \quad \text{for each } (a, b) \in \mathrm{E},
$$
then it is written as $\mathrm{E} = \mathrm{E}^{-1}$.
\end{remark}

\begin{definition}\cite{lewandowska1, lewandowska}
Let $A$ be a vector space over $\mathbb{C}$ and $\mathrm{E} \subset A\times A$ a set with the property $\mathrm{E} = \mathrm{E}^{-1}$. A aplication $\|\cdot,\cdot\|\colon \mathrm{E} \longrightarrow [0, \infty)$ that satisfies:
\begin{enumerate}
\item[(S1)] $\|a,b\| = \|b,a\|$ for each $(a,b) \in \mathrm{E}$;

\item[(S2)] $\|x,\alpha y\| = |\alpha|\|x,y\|$ for each $\alpha \in \mathbb{C}$ and $(a,b) \in \mathrm{E}$;

\item[(S3)] $\|x,y+z\| \leq \|x,y\| + \|x,z\|$ for $a,b,c \in X$ such that $(a,b), (a,c) \in \mathrm{E}$;
\end{enumerate}
is known as a \emph{generalized symmetric two-norm on $\mathrm{E}$}.\\
In the case where $\mathrm{E}$ coincides with $ A\times A$, we will refer to the mapping $\|\cdot,\cdot\|$ simply as a generalized symmetric two-norm on $A$ and to the pair $(A,|\cdot , \cdot\|)$ a generalized symmetric two-normed space.
\end{definition}

\begin{proposition}\label{simgen}\cite{lewandowska1}
Any two-normed space in the Gähler sense can be viewed as a generalized symmetric two-normed space.
\end{proposition}

Next, we will show a direct process by which a generalized two-normed space is guaranteed from two semi-normed spaces.
\begin{proposition}\cite{lewandowska}\label{seminormados}
Let $(A, \|\cdot\|_{A})$, $(B, \|\cdot\|_{B})$ be complex semi-normed spaces. Then, the function $\|\cdot,\cdot\|\colon A\times B\longrightarrow [0,\infty )$ defines a generalized two-norm through the expression
\begin{align*}
\|a,b\|=\|a\|_{A}\|b\|_{B}.
\end{align*}
\end{proposition}

\begin{proposition}\label{consecuencia triangular}
\cite{white} If $(a, c), (b, c)$ are elements of the two-normed space $(\mathbb{W} \subset A\times B, \|\cdot, \cdot\|)$, then
$$\big| \|a,c\|-\|b,c\| \big|\le  \|a-b,c\|$$
\end{proposition}

\begin{definition}
\cite{lewandowska1} A sequence $\{a_n\}_{n \in \mathbb{N}}$ in the generalized two-normed space $(\mathbb{W} \subset A\times B, \|\cdot, \cdot\|)$ converges to an element $a$ in $A$ if for all $\epsilon >0$, there exists $N\in \mathbb{N}$ such that $\|a_n - a, k\|<\epsilon$ whenever $n>N$ is verified.
\end{definition}

\begin{definition}
\cite{lewandowska1} A sequence $\{a_n\}_{n \in \mathbb{N}}$ in the generalized two-normed space\\ $(\mathbb{W} \subset A\times B, \|\cdot, \cdot\|)$ is said \emph{$k$-Cauchy} if for all $\epsilon >0$, there exists a natural number $N$ such that $\|a_m - a_n , k\|<\epsilon$ whenever $m,n>N$ is verified.
\end{definition}

\begin{definition}
A \emph{generalized $(2, k)$-Banach space} is that space in which any $k$-Cauchy sequence converges to an element of the space.
\end{definition}

\section{Main results}
  
\subsection{Bounded generalized 2-variation}

In the present section we construct a type of variation for definite mappings of a closed interval of real numbers in a generalized two-normed space.
\begin{definition}
Let $(\mathbb{W}\subset A\times B,\|\cdot,\cdot\|)$ be a generalized two-normed space and $k\in B$. 
Let us consider a mapping $g\colon [a,b]\subset\mathbb{R}\longrightarrow \mathbb{W}^{k}$ and  $P\in\mathcal{P}([a,b])$ a partition.

We will call the sum $\displaystyle\sum_{i=1}^{n} \| g(t_i)-g(t_{i-1}), k\|$ the 2-variation of $g$ associated with $k$ with respect to the partition $P$ in the generalized two-normed space $(\mathbb{W},\|\cdot,\cdot\|)$. We will write the above by $V_a^b[g,\mathbb{W}^{k};P]$.

The supremum of the set $\left\{ V_a^b[g,\mathbb{W}^{k};P]\colon P\in\mathcal{P}([a,b])\right\}$, assuming that this exists, will be called \emph{2-variation of $g$ associated with $k$ in the generalized two-normed space $(\mathbb{W},\|\cdot,\cdot\|)$}, 
and we will write it by $V_{a}^{b}(g,\mathbb{W}^{k})$.
\end{definition}
It is clear that the non-negativity of $V_a^b(g,\mathbb{W}^{k})$ is guaranteed.

\begin{remark}
Note that since $\mathbb{W}^{k}$ is a vector subspace of $X$, $(g(s)-g(t)) \in \mathbb{W}^{k}$ for all $s,t \in [a,b]$.
\end{remark}

\begin{example}\label{ejemplo2ka}
Consider $(\mathbb{W} \subset \mathbb{C}^2 \times \mathbb{C}, \|\cdot, \cdot\|)$ with the generalized two-norm given in \ref{generaltwo-norma}. Let $\sqrt{2} \in \mathbb{C}$ be fixed and $g \colon [a,b] \longrightarrow \mathbb{W}^{\sqrt{2}}$ be a function defined by $g(x) = (xi, xi)$. We will now demonstrate that 
$g$ possesses bounded generalized 2-variation.

Indeed, let $P = \left\{x_0, \cdots, x_n\right\} \in \mathcal{P}([a,b])$, then
\begin{align*}
\sum_{j=1}^{n}\|g(x_{j})-g(x_{j-1}),\sqrt{2}\|&=\sum_{j=1}^{n}\|(ix_{j},ix_{j})-(ix_{j-1},ix_{j-1}),\sqrt{2}\|\\
&=\sum_{j=1}^{n}\|(ix_{j}-ix_{j-1},ix_{j}-ix_{j-1}),\sqrt{2}\|\\
&=\sqrt{2}\sum_{j=1}^{n}\sqrt{|ix_{j}-ix_{j-1}|^{2}+|ix_{j}-ix_{j-1}|^{2}}\\
&=\sqrt{2}\sum_{j=1}^{n}\sqrt{|x_{j}-x_{j-1}|^{2}+|x_{j}-x_{j-1}|^{2}}\\
&=2\sum_{j=1}^{n}\sqrt{|x_{j}-x_{j-1}|^{2}}=2\sum_{j=1}^{n}(x_{j}-x_{j-1})\\
&=2(b-a) .
\end{align*}
It is concluded that the set $\{ V^{a}_{b}[g,\mathbb{W}^{\sqrt{2}};P] \colon P \in \mathcal{P}([a,b]) \}$ is upper bounded by $2(b-a)$. Thus verifying that $g$ is a bounded 2-variance function in $(\mathbb{W} \subset \mathbb{C}^2 \times \mathbb{C}, \|\cdot, \cdot\|)$.
\end{example}

\begin{theorem}\label{ass}
Given a {two-normed space} $(\mathbb{W}\subset A\times B,\|\cdot,\cdot\|)$, $k,l\in B$ fixed, $\alpha \in \mathbb{C}$ and two applications $g,h\colon [a,b]\to \mathbb{W}^{k}$, are verified:
\begin{enumerate}
\item[$(i)$] $V_a^b(\alpha g,\mathbb{W}^{k})=| \alpha | V_a^b(g,\mathbb{W}^k)$.
\item[$(ii)$] $V_a^b(g+h,\mathbb{W}^k)\leq V_a^b(g,\mathbb{W}^{k})+V_a^b(h,\mathbb{W}^{k})$.
\item[$(iii)$] $V_a^b(g,\mathbb{W}^{k}\cap \mathbb{W}^{l})\leq V_a^b(g,\mathbb{W}^{k})+V_a^b(g,\mathbb{W}^{l})$.
\end{enumerate}
\end{theorem}

\begin{proof}
(i) and (ii) are a consequence of $\mathbb{W}^{k}$ being a vector subspace of $A$.

(iii) Let $g \colon [a,b] \longrightarrow \mathbb{W}^{k} \cap \mathbb{W}^{l}$. Thus,
\begin{align*}
V_a^b(g,\mathbb{W}^{k}\cap \mathbb{W}^{l})&=\sup \left\{ \sum_{i=1}^{n} \| g(t_i)- g(t_{i-1}), p\|\colon p\in \mathbb{W}^{k}\cap \mathbb{W}^{l}\right\}\\
&\le \sup \left\{ \sum_{i=1}^{n} \| g(t_i)- g(t_{i-1}), p\|\colon p\in \mathbb{W}^{k}\cap \mathbb{W}^{l}\right\}\\
&\ \ \ +\sup \left\{ \sum_{i=1}^{n} \| g(t_i)- g(t_{i-1}), p\|\colon p\in \mathbb{W}^{k}\cap \mathbb{W}^{l}\right\}\\
&\le V_a^b(g,\mathbb{W}^{k})+V_a^b(g,\mathbb{W}^{l})
\end{align*}
\end{proof}

\subsection{Applications with bounded generalized 2-variation}

From these definitions, we are now able to define what an application with bounded generalized 2-variation is.
\begin{definition}\label{vaacotada}
Consider a generalized two-normed space $(\mathbb{W}\subset A\times B,\|\cdot,\cdot\|)$ and $k\in B$ fixed. We refer to an application $g\colon [a,b]\to \mathbb{W}^{k}$ \emph{of bounded generalized 2-variation} if the set $$\left\{ V_a^b[g,\mathbb{W}^{k};P]\colon P\in\mathcal{P}([a,b])\right\}$$ is upper bounded. 
\end{definition}

\begin{remark}
The notation $BV([a, b], \mathbb{W}^k)$ will be used to represent the family of functions defined on $[a,b]$ that possesing bounded generalized 2-variation, specifically,
$$BV ([a, b],\mathbb{W}^{k})=\left\{g\colon [a,b]\to \mathbb{W}^{k}\colon V_a^b(g,\mathbb{W}^{k})<\infty\right\}.$$
\end{remark}

\begin{remark}\label{2kvar2var}
Let $(X, \|\cdot,\cdot\|)$ be a two-normed space in the Gähler sense. If $g$ is a function with bounded $(2,k)$-variation over $[a,b]$ in $X$, as defined in \cite{CFF}, and $\|\cdot,\cdot\|$ is interpreted as a symmetric generalized two-norm with $\mathbb{W}=X\times X$, then it follows that $V_a^b(g, \mathbb{W}^k) = V_a^b(g, X, k)$. Consequently, $g$ can be classified as a function of bounded generalized 2-variation in $X$.
\end{remark}

The next proposition ensures that, in a semi-normed space, every function of bounded variation belongs to the set of functions with bounded generalized 2-variation, considering the generalized two-norm defined in Proposition \ref{seminormados}.
\begin{proposition}\label{varseminorm}
Let $(X, \|\cdot\|)$ be a semi-normed space and $k \in X$. Consider the symmetric generalized two-norm defined by $\|x,y\| = \|x\|\|y\|$. If $g: [a, b] \longrightarrow (X, \|\cdot\|)$ is a function of bounded variation on $[a, b]$, then $g$ is also of bounded generalized 2-variation in $(X, \|\cdot, \cdot\|)$.
\end{proposition}
\begin{proof}
Given that $g$ has bounded variation in $(X, \|\cdot\|)$, there exists a positive constant $p \in \mathbb{R}^{+}$ such that
\begin{equation}
V_{a}^{b}(g,X)=\sup\left\{\sum_{i=1}^{n}\|g(t_{i})-g(x_{i-1})\| \right\}<p.
\end{equation}
Then,
\begin{align*}
V_{a}^{b}(g,\mathbb{W}^{k}) &=\sup \left\{ \sum_{i=1}^{n} \| g(t_i)-g(t_{i-1}), k\|\right\}=\sup \left\{ \sum_{i=1}^{n} \left\| g(t_i)-g(t_{i-1})\right\|\left\| k\right\|\right\}\\
&=\left\| k\right\|\sup \left\{ \sum_{i=1}^{n} \left\| g(t_i)-g(t_{i-1})\right\|\right\}<\left\| k\right\|p.
\end{align*}
As a result, the boundedness of the generalized 2-variation of the function $g$ in the space $(X, \|\cdot, \cdot\|)$ is guaranteed.
\end{proof}

\begin{proposition}\label{st}
Let $g \colon [a,b] \longrightarrow \mathbb{W}^k$ be a function with bounded generalized 2-variation, where $(\mathbb{W} \subset A \times B, \|\cdot,\cdot\|)$ is a generalized two-normed space. For $s, m \in [a, b]$ with \\$a \leq s < m \leq b$, we have
$$
\|g(s) - g(m), k\| \leq V_a^b(g, \mathbb{W}^k).
$$
\begin{proof}
Let $P_{0}=\{a, s, m, b\}$ be a partition of $[a, b]$.\\
Then, since $V_{a}^{b}[g,\mathbb{W}^{k};P_0] \in \left\{ V_{a}^{b}[g,\mathbb{W}^{k};P]\colon P\in\mathcal{P}([a,b])\right\}$, we have
\begin{align*}
V_{a}^{b}(g,\mathbb{W}^{k})&=\sup \bigg\{ V_{a}^{b}[g,\mathbb{W}^{k};P]\colon P\in\mathcal{P}([a.b])\bigg\}\ge V_{a}^{b}[g,\mathbb{W}^{k};P_0]\\
&=\|g(s)-g(a),k\|+\left\|g(m)-g(s),k\right\|+\|g(b)-g(m),k\|\\
&\ge \left\|g(m)-g(s),k\right\|.
\end{align*}
In this way,
$$
\|g(s)-g(m),k\| \leq V_{a}^{b}(g,\mathbb{W}^{k})\text{ for each $a \leq s<m \leq b$.}
$$
\end{proof}
\end{proposition}

The following theorem establishes that any function with bounded generalized 2-variation in a generalized two-normed space is also $(2,k)$-bounded.
\begin{theorem}\label{varimplica2kacotada}
In a generalized two-normed space $(\mathbb{W} \subset A \times B, \|\cdot, \cdot\|)$ with $k \in B$, it holds that $g \colon [a,b] \longrightarrow \mathbb{W}^k$ is $(2,k)$-bounded whenever $g$ has bounded generalized 2-variation.
\end{theorem}
\begin{proof}
Given that $g$ has bounded generalized 2-variation, there exists a constant $\sigma \in \mathbb{R}^{+}$ such that
\begin{equation}\label{inf1}
V_{a}^{b}(g,\mathbb{W}^{k})\le \sigma.
\end{equation}
Let $x \in (a, b)$. Consider the partition $P_{0}=\{ a,x,b\}$. Note that
$$
\|g(x)-g(a),k\|\leq \|g(x)-g(a),k\|+\|g(b)-g(x),k\|=V_{a}^{b}[g,\mathbb{W}^{k};P_{0}].
$$
Furthermore, using Proposition \ref{consecuencia triangular}, we have
\begin{equation}\label{inf2}
\left\|g(x),k\right\|-\left\|g(a),k\right\|\le\big| \left\|g(x),k\right\|-\left\|g(a),k\right\| \big|\le  \left\|g(x)-g(a),k\right\|.
\end{equation}
By applying the inequalities \eqref{inf1} and \eqref{inf2}, it follows that
\begin{align*}
\sigma&\ge V_{a}^{b}(g,\mathbb{W}^{k})\ge V_{a}^{b}[g,\mathbb{W}^{k};P_{0}]= \|g(x)-g(a),k\|+\|g(b)-g(x),k\|\\
&\ge \|g(x)-g(a),k\|\ge \left\|g(x),k\right\|-\left\|g(a),k\right\|
\end{align*}
Consequently,
\begin{equation}\label{inf3}
\left\|g(x),k\right\|\le \left\|g(a),k\right\|+\sigma\le \left\|g(a),k\right\|+\|g(b),k\|+\sigma
\end{equation}
for all $x \in [a, b]$. Thus, the function $g$ is $(2,k)$-bounded on the interval $[a,b]$.
\end{proof}

The converse of the aforementioned theorem does not hold in all cases, as demonstrated by the following counterexample.
\begin{example}\label{ejemplo 1k}

Consider the space $\left(\mathbb{W} \subset \mathbb{C}^2 \times \mathbb{C}, \|\cdot, \cdot\|\right)$, where the generalized two-norm is specified in Example \ref{generaltwo-norma} as 
$$
\|z, k\| := \sqrt{|z_1|^2 + |z_2|^2} |k|.
$$
Let $k=-\frac{1}{2}i\in \mathbb{C}$ be. Let us see that $g\colon [\sqrt{3},2]\to\mathbb{W}^{-\frac{1}{2}i}$ defined by
$$
g(x)=\left\{
\begin{matrix}
(2i,-i),\ \text{if}\ x\ \text{is rational},\ x\in[\sqrt{3},2], \\
 \\ \ \ (0,0),\ \text{if}\ x\ \text{is irrational},\ x\in[\sqrt{3},2].
 \end{matrix}\right.
 $$

Let us verify that $g$ is $(2,k)$-bounded, but not of bounded generalized 2-variation.

Indeed, if $x$ is rational,
$$
\|g(x),-\frac{1}{2}i\|=\|(2i,-i),-\frac{1}{2}i\|=\sqrt{|2i|^{2}+|-i|^{2}} |-\frac{1}{2}i|=\frac{\sqrt{5}}{2}.
$$
If $x$ is irrational,
$$
\|g(x),-\frac{1}{2}i\|=\|(0,0),-\frac{1}{2}i\|=\sqrt{|0|^{2}+|0|^{2}} |-\frac{1}{2}i|=0.
$$
Hence, 
$$
\|g(x),-\frac{1}{2}i\|=\left\{\begin{matrix}
\frac{\sqrt{5}}{2},\ \text{if}\ x\ \text{is rational},\ x \in[\sqrt{3},2], \\
 \\ \ \ 0 ,\ \text{if}\ x\ \text{is irrational},\ x\in[\sqrt{3},2].
\end{matrix}\right.
$$ 
Therefore, $\|g(x),-\frac{1}{2}i\|\leq \frac{\sqrt{5}}{2}$ for all $x \in [\sqrt{3},2].$ Thus, $g$ is $(2,k)$-bounded in $[\sqrt{3},2]$.

Next we will show an application $g$ described above does not satisfy the condition of bounded variability in a generalized two-normed space.

We will construct a partition of the interval $[\sqrt{3}, 2]$ as follows:

We take as first element the irrational $\sqrt{3}$, second element a rational number and so on we will alternate between rationals and irrationals until we get to the rational 2.

Then,
\begin{align*}
V_{\sqrt{3}}^{2}\left(g,\mathbb{W}^{-\frac{1}{2}i}\right)&=\sup\Bigg\{\sum_{j=1}^{n}\| g(t_j)-g(x_{j-1}), k\|:P\in\mathcal{P}[a,b]\Bigg\}\\
&\ge \sum_{j=1}^{n}\left\|g(x_j)-g(x_{j-1}),-i\right\|\\
&=\|g(x_1)-g(x_0),-i\|+\cdots+\|g(x_n)-g(x_{n-1}),-i\|\\
&=\|(2i,-i),-i\|+ \cdots+\|(2i,-i),-i\| \\
&=\frac{\sqrt{5}}{2} + \cdots + \frac{\sqrt{5}}{2}= n\frac{\sqrt{5}}{2}.
\end{align*}

In the partition of $[\sqrt{3},2]$ constructed above it is evident that the set $\left\{ V_a^b[g,\mathbb{W}^{k};P]\colon P\in\mathcal{P}([a,b])\right\}$ is not upper bounded, which implies that the application $g$ does not possess bounded generalized 2-variation.

\end{example}

In the following section, we show that the set of functions under the definition given in \ref{vaacotada} turns out to be a vector space.

\subsection{New applications of bounded generalized 2-variation}

\begin{theorem}\label{sum prod}
For any pair of applications on a generalized two-normed space and a scalar, it is satisfied that both the scalar multiple and the sum are also functions of bounded generalized 2-variation.
\end{theorem}
\begin{proof}
If $g, h \in BV([a,b], \mathbb{W}^{k})$, then there exist $\lambda, \gamma \in \mathbb{R}^{+}$ such that
$$
V_{a}^{b}(g,\mathbb{W}^{k}) \le \lambda \quad \text{ and } \quad V_{a}^{b}(h,\mathbb{W}^{k}) \le \gamma.
$$
Then, taking $\left| \alpha \right|\lambda = M$ and $\lambda + \gamma = N$, and using Theorem \ref{ass}, we obtain
$$
V_{a}^{b}(\alpha g,\mathbb{W}^{k}) \le M \quad \text{ and } \quad V_{a}^{b}(g+h,\mathbb{W}^{k}) \le N.
$$
Thus it is evident that both scalar multiples and sums of applications of bounded generalized 2-variation preserve the property of bounded generalized 2-variation.
\end{proof}
 \begin{theorem}\label{distributive}
If $g\in BV([\alpha ,\beta ],\mathbb{W}^{k})$ and $\lambda \in [\alpha , \beta]$, then it is satisfied that $$V_{\alpha}^{\lambda}(g,\mathbb{W}^{k})+V_{\lambda}^{\beta}(g,\mathbb{W}^{k})\le V_{\alpha}^{\beta}(g,\mathbb{W}^{k}).$$ 
\end{theorem}
\begin{proof}
It is straightforward to verify that $V_{\alpha}^{\lambda}(g, \mathbb{W}^{k})$ and $V_{\lambda}^{\beta}(g, \mathbb{W}^{k})$ are both finite.

Let $\epsilon > 0$ be given.\\ We take partitions $P_{[\alpha,\lambda]} = \{t_0, t_1, \cdots, t_j\} \in \mathcal{P}([\alpha, \lambda])$ and $P_{[\lambda,\beta]} = \{t_j, t_{j+1}, \cdots, t_n\} \in \mathcal{P}([\lambda, \beta])$ such that $t_{n} = \lambda$,
\[
V_{\alpha}^{\lambda }(g, \mathbb{W}^{k}) - \frac{1}{2}\epsilon \leq \sum_{i=1}^j \|g(t_i) - g(t_{i-1}), k\|
\]
and
\[
V_{\lambda}^{\beta}(g, \mathbb{W}^{k}) - \frac{1}{2}\epsilon \leq \sum_{p=j+1}^n \|g(t_p) - g(t_{p-1}), k\|.
\]
Then,
\[
V_{\alpha}^{\lambda }(g, \mathbb{W}^{k}) + V_{\lambda}^{\beta}(g, \mathbb{W}^{k}) - \epsilon
\leq \sum_{i=1}^n \|g(t_p) - g(t_{p-1}), k\|
\leq V_{\alpha}^{\beta }(g, \mathbb{W}^{k}).
\]
Thus, it is proven.
\end{proof}

Now, we guarantee the existence of a generalized symmetric two-norm for any function space with bounded generalized 2-variation.
\begin{theorem}\label{two-norma}
Given a generalized two-normed space $(\mathbb{W} \subset A\times B, \|\cdot,\cdot\|)$. The function $$\|\cdot,\cdot\|_{2G}\colon BV ([a, b],\mathbb{W}^{k})\times BV ([a, b],\mathbb{W}^{k}) \longrightarrow [0,\infty )$$ defined by
$$
\|f,h\|_{2G}:=V_{a}^{b}(f,\mathbb{W}^{k})  \left\|h(a),k\right\|+V_{a}^{b}(h,\mathbb{W}^{k}) \left\|f(a),k\right\|
$$
is a generalized symmetric two-norm on $BV ([a, b], \mathbb{W}^{k})$.
\begin{proof}
(S1)
\begin{align*}
\|f,h\|_{2G}&=V_{a}^{b}(f,\mathbb{W}^{k})  \left\|h(a),k\right\|+V_{a}^{b}(h,\mathbb{W}^{k}) \left\|f(a),k\right\| \\
&=V_{a}^{b}(h,\mathbb{W}^{k}) \left\|f(a),k\right\|+V_{a}^{b}(f,\mathbb{W}^{k})  \left\|h(a),k\right\|=\|h,f\|_{2G}
\end{align*}

(S2)
\begin{align*}
\|\alpha f,h\|_{2G}&=V_{a}^{b}(\alpha f,\mathbb{W}^{k})  \left\|h(a),k\right\|+V_{a}^{b}(h,\mathbb{W}^{k}) \left\|\alpha f(a),k\right\| \\
&=\left| \alpha\right|V_{a}^{b}( f,\mathbb{W}^{k})  \left\|h(a),k\right\|+\left| \alpha\right|V_{a}^{b}(h,\mathbb{W}^{k}) \left\|f(a),k\right\|  =\left|\alpha\right| \|f,h\|_{2G}
\end{align*}

(S3)
\begin{align*}
\|f,h\|_{2G}+\|g,h\|_{2G}&=V_{a}^{b}(f,\mathbb{W}^{k})  \left\|h(a),k\right\|+V_{a}^{b}(h,\mathbb{W}^{k}) \left\|f(a),k\right\|\\
&\ \ \ + V_{a}^{b}(g,\mathbb{W}^{k})  \left\|h(a),k\right\|+V_{a}^{b}(h,\mathbb{W}^{k}) \left\|g(a),k\right\| \\
&= \left( V_{a}^{b}(f,\mathbb{W}^{k}) + V_{a}^{b}(g,\mathbb{W}^{k})\right)\left\|h(a),k\right\|\\
&\ \ \ +V_{a}^{b}(h,\mathbb{W}^{k}) \left( \left\|f(a),k\right\|+\left\|g(a),k\right\|\right)\\
&\ge  V_{a}^{b}(f+g,\mathbb{W}^{k})\left\|h(a),k\right\| +V_{a}^{b}(h,\mathbb{W}^{k}) \left\|f(a)+g(a),k\right\|\\
&=\|f+g,h\|_{2G}.
\end{align*}
The above allows us to state that every space of functions of bounded variation is in fact a generalized symmetric two-normed space.
\end{proof}
\end{theorem}
Note that in Example \ref{ejemplo2ka},
\begin{align*}
\| f,f\|_{2G}&=2 V_{a}^{b}(f,\mathbb{W}^{\sqrt{2}})\|f(a),\sqrt{2}\|=2(2(b-a))\|f(a),\sqrt{2}\|\\
&=4(b-a)\sqrt{|ai|^{2}+|ai|^{2}}|\sqrt{2}|=8|a|(b-a).
\end{align*}

\section{Conclusions}

In semi-normed spaces, the concept of a function with bounded variation naturally extends to a function of bounded generalized variation (Proposition \ref{varseminorm}). Furthermore, the notion of bounded variation is well-defined in generalized two-normed spaces (Definition \ref{vaacotada}).  
Thus, the notion presented in this research generalizes the one introduced by Chistyakov \cite{Chistyakov Metric}. Moreover, Remark \ref{2kvar2var} shows that the results given by Cure, Ferrer S., and Ferrer V. \cite{CFF} are a consequence of the findings of this work. In generalized two-normed spaces, the variation of a function determines an upper bound on the distance between the images of functions with bounded variation under the generalized two-norm (Proposition \ref{st}). Additionally, the space of functions exhibiting generalized bounded variation in such spaces can be regarded as a generalized symmetric two-normed space (Theorem \ref{two-norma}).

\section*{Author contributions}
Ferrer Kandy, Cure Jaffeth and Ferrer Osmin: Conceptualization, Methodology, Validation, Software,
Writing-original draft, Writing-review \& editing. All authors contributed equally to the manuscript.
All authors have read and approved the final version of the manuscript for publication.
\subsection*{Acknowledgements}
This research was supported by the University of Sucre, Sincelejo-Colombia, the Pontifical Bolivarian University, Monter\'ia-Colombia and the Catholic University of the Most Holy Conception, Concepci\'on-Chile.
\section*{Conflict of interest}
The authors declare that they have no conflict of interest in this work.

\end{document}